\documentclass[11pt,letterpaper,reqno]{amsart}
\usepackage[margin=1in]{geometry}
\usepackage{color}
\usepackage[latin1]{inputenc}
\usepackage{amsmath,amsfonts,amsthm,epsfig,graphicx,amssymb,bm}
\usepackage[colorlinks=true, linkcolor=blue, citecolor=blue]{hyperref}
\usepackage{cleveref}
\usepackage{esint}
\usepackage{caption}
\usepackage{enumitem}
\usepackage{overpic}

\newcommand{\F}{\mathcal F}

\renewcommand{\H}{\mathcal{H}}

\newcommand{\R}{\mathbb{R}}
\renewcommand{\SS}{\mathbb{S}}

\newcommand{\Om}{\Omega}

\newcommand{\g}{\gamma}

\newcommand{\s}{\sigma}

\newcommand{\dist}{{\rm dist}}

\newcommand{\spt}{{\rm spt}}

\newcommand{\weakstar}{\stackrel{*}{\rightharpoonup}}

\newcommand{\pa}{\partial}

\newcommand{\cc}{\subset\!\subset}

\newcommand{\cl}{\mathrm{cl}\,}

\newcommand{\D}{\mathcal{D}}

\newcommand{\var}{\mathbf{var}\,}

\newcommand{\C}{\mathcal{C}}

\newcommand{\ehn}{\overset{\mathcal{H}^n}{=}}

\newcommand\restr[2]{{
  \left.\kern-\nulldelimiterspace 
  #1 
  \right|_{#2} 
  }}
\newcommand{\K}{\mathcal{K}}

\newcommand{\RR}{\mathcal{R}}

\newcommand{\one}{{\scriptscriptstyle{(1)}}}
\newcommand{\zero}{{\scriptscriptstyle{(0)}}}

\newcommand{\kkms}{\mathcal{K}}

\newcommand{\shn}{\overset{\mathcal{H}^n}{\subset}}

\newcommand{\mres}{\mathbin{\vrule height 1.6ex depth 0pt width 
0.13ex\vrule height 0.13ex depth 0pt width 1.3ex}}

\theoremstyle{plain}
\newtheorem{theorem}{Theorem}[section]

\newtheorem{proposition}[theorem]{Proposition}
\newtheorem*{theorem*}{Theorem}
\newtheorem*{corollary*}{Corollary}

\theoremstyle{definition}
\newtheorem{definition}[theorem]{Definition}

\newtheorem{remark}[theorem]{Remark}

\newtheorem*{notation*}{Notation}

\numberwithin{equation}{section}
\numberwithin{figure}{section}

\newcommand{\wire}{\mathbf{W}}

\allowdisplaybreaks

\title{On the relaxation of Gauss's capillarity theory under spanning conditions}

\author{Michael Novack}
\address{Department of Mathematical Sciences, Carnegie Mellon University, Wean Hall 6113, Pittsburgh, PA 15213, United States of America}
\email{mnovack@andrew.cmu.edu}

\begin{document}

\begin{abstract} 
We study a variational model for soap films in which the films are represented by sets with fixed small volume rather than surfaces. In this problem, a minimizing sequence of completely ``wet" films, or sets of finite perimeter spanning a wire frame, may converge to a film containing both wet regions of positive volume and collapsed (dry) surfaces. When collapsing occurs, these limiting objects lie outside the original minimization class and instead are admissible for a relaxed problem. Here we show that the relaxation and the original formulation are equivalent by approximating the collapsed films in the relaxed class by wet films in the original class.
\end{abstract}

\maketitle

\setcounter{tocdepth}{2}


\section{Introduction}\label{sec:intro}
We analyze a model for soap films based on the classical Gauss free energy functional from capillarity theory \cite[Ch. 1.4]{Finn}. In this model, which was proposed by King-Maggi-Stuvard in \cite{KMS1}, one minimizes the surface tension energy among sets with small volume $v$ that satisfy a spanning condition with respect to a wire frame $\wire\subset \mathbb{R}^{n+1}$, the complement of which is the container accessible to the soap. Informally, the problem is
\begin{align}\label{intro soap film cap}
    \inf \{\mathcal{H}^n(\pa E \setminus \wire) : E \subset \mathbb{R}^{n+1}\setminus \wire\,,\,\,|E|=v\,,\mbox{ $\pa E$ spans $\wire$}  \}\,.
\end{align}
In this paper we use the notion of ``spanning" from \cite{MNR1}, which generalizes the idea of Harrison-Pugh \cite{HP16}; see Section 2 for a complete discussion. The addition of the volume constraint adds a length scale to the minimal surface model, which is recovered in the vanishing volume limit. The soap film capillarity model is thus well-suited for describing some features of real films that cannot be captured by minimal surfaces. Such features include the thickened tubes of liquid ``wetting" a line of $Y$-point singularities, which play an important role in behaviors such as drainage and are known in the physical literature as Plateau borders; see e.g. \cite{foamdrainageI, foamdrainageII} and the books \cite{weaireBOOK} and \cite[Ch. 2]{foamchapter}.

An important aspect of the model is the phenomenon of collapsing, which occurs in regions where it is energetically convenient for the (multiplicity one) boundaries of a minimizing sequence $\{E_j\}_j$ to collapse onto a multiplicity two surface; see Figure \ref{collapsed triple junction} below. Collapsing presents mathematical challenges stemming 
from the fact that the limit of a minimizing sequence for \eqref{intro soap film cap}, or generalized minimizer, may not belong to the original class, but rather to a relaxed class also including collapsed competitors. This complicates the study of properties such as regularity and the wetting of singularities, since the convergence of the minimizing sequence to the generalized minimizer only enables one to test the minimality amongst objects which also arise as limits of sets of finite perimeter and not against the entire relaxed class.
\par
Our main goal is to show that \eqref{intro soap film cap} and its relaxation to collapsed films are in fact equivalent minimization problems. Concise statements of the main results in this article are as follows:
\begin{enumerate}[label=(\roman*)]
    \item every collapsed competitor in the relaxed class can be approximated by a sequence of non-collapsed competitors with the same volume (Theorem \ref{thm main of approximation});
    \item generalized minimizers arising as limits in \eqref{intro soap film cap} minimize the relaxed energy functional among the entire relaxed class (Theorem \ref{thm main of minimality}).
\end{enumerate}
The strengthened minimality of (ii) simplifies the mathematical investigation of collapsing and wetting phenomena in the soap film capillarity model.

\par
The paper is organized as follows. Section \ref{sec:background} contains the necessary background and precise statements of our results. After collecting some preliminaries in Section \ref{sec:prelim}, we then prove Theorems \ref{thm main of approximation}-\ref{thm main of minimality} in Section \ref{sec:proofs}.

\begin{figure}
\begin{overpic}[scale=0.6,unit=1mm]{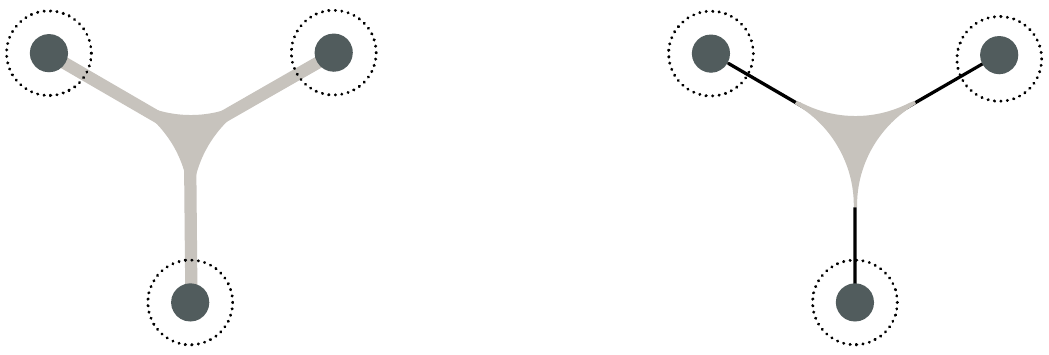}
    \put(60,25){\small{$\gamma_1$}}
    \put(-3,25){\small{$\gamma_1$}}
    \put(36.5,25){\small{$\gamma_2$}}
    \put(100,25){\small{$\gamma_2$}}
    \put(23,4){\small{$\gamma_3$}}
    \put(86.5,4){\small{$\gamma_3$}}
    \put(32,12){\small{$\mathcal{C}=\{\gamma:\exists \, i \,\mbox{ s.t. }[\gamma]=[\gamma_i]\}$}}
    \put(12,17){\small{$E'$}}
    \put(76,17){\small{$E$}}
    \put(82,10){\small{$K \setminus \partial^* E$}}
\end{overpic}
\caption{On the left is a non-collapsed set $E'$ approximating the collapsed minimizer on the right. $K\setminus \partial^* E$ is the three black segments.}\label{collapsed triple junction}
\end{figure}
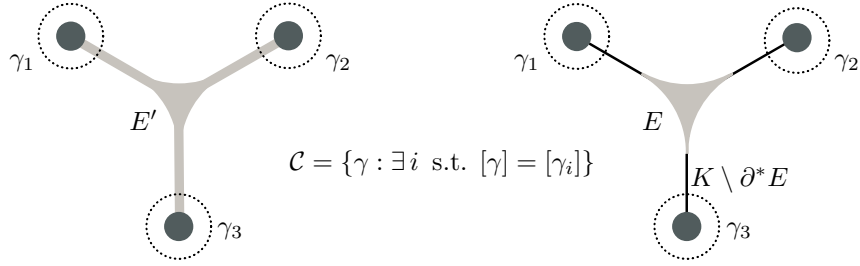

\section{Statements}\label{sec:background}
\subsection{Spanning and the Plateau problem}

In order to formulate any Plateau-type problem, one must make mathematically rigorous the notion that a soap film ``spans a given wire frame.'' The approach taken here is based on a generalization by F. Maggi, D. Restrepo, and the author \cite{MNR1} of the idea of Harrison-Pugh \cite{HP16}. Among the various versions of Plateau's problem, a key feature of the Harrison-Pugh version is that it leads to minimizers that exhibit the physical singularities predicted by Plateau and validated mathematically in \cite{Taylor}. {\color{black}We mention that \cite{HP16} has spurred much recent progress on the Plateau problem. More generally, Harrison-Pugh \cite{HP17} proved an existence result in arbitrary dimension/co-dimension for the anisotropic Plateau problem in a large class of ambient spaces and encompassing several different spanning conditions inside a general axiomatic framework. We refer the reader also to \cite{DGM,DPDRG,HP16b,DLDRG,HP17,DR,FangKo,DPDRG2} and the references therein for additional works in this vein.} 

Let $n\geq 0$, fix closed $\wire\subset \mathbb{R}^{n+1}$ (the wire frame), and set $\Omega=\mathbb{R}^{n+1}\setminus \wire$. Following the presentation of the Harrison-Pugh spanning condition from \cite{DGM}, we define a {\bf spanning class} $\mathcal{C}$ to be a non-empty family of smooth embeddings of $\mathbb{S}^1$ into $\Omega$ which is closed by homotopy in $\Omega$. 

\begin{definition}[Homotopic spanning for closed sets]\label{def homotop span top}
A relatively closed subset $S$ of $\Omega$ is $\mathcal{C}${\bf -spanning} $\wire$ if $S \cap \gamma \neq \varnothing$ for all $\gamma \in \mathcal{C}$. 
\end{definition}

\noindent The corresponding formulation of the Plateau problem is
\begin{align}\label{plateau problem}
    \ell := \inf\{\mathcal{H}^n(S): \mbox{$S\subset \Omega$ is relatively closed and $\mathcal{C}$-spanning $\wire$} \}.
\end{align}
Compactness and lower-semicontinuity for minimizing sequences yield the existence of a minimizer \cite{HP16, DGM}. In these arguments, the asymptotic minimality of the sequence is utilized to make energy comparisons that yield a limiting surface satisfying Definition \ref{def homotop span top}. In scenarios such as the proof of the approximation item (i) from the introduction or the Allen-Cahn setting of \cite{MNR2}, it is useful to have a definition of spanning compatible with energy-bounded but non-minimizing sequences. Such a framework was developed in \cite{MNR1}, which we now recall.
\par
If $\mathcal{C}$ is a spanning class for a closed set $\wire$, we define the {\bf tubular spanning class} $\mathcal{T}(\mathcal{C})$ associated to $\mathcal{C}$ to be the family of triples $(\gamma, \Phi, T)$ such that $\g\in\C$, $T = \Phi(\mathbb{S}^1 \times B_1^n)$, and
  \[
  \textup{$\Phi:\mathbb{S}^1 \times \cl B_1^n \to  \Omega$ is a diffeomorphism with $\restr{\Phi}{\mathbb{S}^1\times \{0\}}= \gamma$}\,.
  \]
  When $(\gamma, \Phi, T)\in\mathcal{T}(\mathcal{C})$, the slice of $T$ defined by $s\in\mathbb{S}^1$ is
  \[
  T[s]=\Phi(\{s\}\times B_1^n)\,.
  \]
We will need the measure theoretic notion of connectedness from \cite{CCDPMSteiner,CCDPM17}. For a Borel set $G$, we let $G^{(t)}$, $t\in [0,1]$, be the points of Lebesgue density $t$ and $\partial^e G= (G^{(1)} \cup G^{(0)})^c$. For Borel sets $S$, $G$, $G_1$ and $G_2$ in $\mathbb{R}^{n+1}$, $S$ {\bf essentially disconnects} $G$ into $\{G_1,G_2\}$ if
\begin{equation}\notag
|G\Delta(G_1\cup G_2)|=0\,,\quad|G_1||G_2|>0\,,\quad\textup{and}\quad
G^{(1)}\cap \partial^e G_1 \cap \partial^e G_2\shn S \,.
\end{equation}
Here $|\cdot|$ is the Lebesgue measure $\mathcal{L}^{n+1}$ and $A\shn B$ means $\mathcal{H}^n(A \setminus B)=0$; in words, ``$A$ is $\mathcal{H}^n$-contained in $B$." 

\begin{definition}[Measure-theoretic spanning]\label{def homot span borel}
    Given a closed set $\wire$ and a spanning class $\mathcal{C}$ for $\wire$, a Borel set $S \subset\Omega$ is {\bf $\mathcal{C}$-spanning} $\wire$ if for each $(\gamma,\Phi, T)\in \mathcal{T}(\mathcal{C})$, $\mathcal{H}^1$-a.e. $s\in \mathbb{S}^1$ has the following property: 
    \begin{eqnarray}\notag
      &&\mbox{for $\mathcal{H}^n$-a.e. $x\in T[s]$, $\exists$ a partition $\{T_1,T_2\}$ of $T$ with $x\in\partial^e T_1 \cap \partial^e T_2$}\,,\\ \nonumber
      &&\quad\mbox{and such that $S \cup  T[s]$ essentially disconnects $T$ into $\{T_1,T_2\}$}\,.
    \end{eqnarray}
\end{definition}

\noindent Associated to this concept is the Plateau problem with measure-theoretic spanning
\begin{align}\notag
    \ell_{\rm B} := \inf\{\mathcal{H}^n(S): \mbox{$S\subset \Omega$ is 
Borel and $\mathcal{C}$-spanning $\wire$} \}\,.
\end{align}
Definition \ref{def homot span borel} is an acceptable generalization of Definition \ref{def homotop span top} for two reasons.
\begin{itemize}
\item {\it Equivalence of spanning definitions:}
If $S$ is relatively closed in $\Omega$, then $S$ satisfies Definition \ref{def homotop span top} if and only if it satisfies Definition \ref{def homot span borel} \cite[Theorem A.1]{MNR1}. 

\item {\it Equivalence of Plateau problems: } If $\ell<\infty$, then $\ell_{\rm B}=\ell$, and any minimizer for $\ell_{\rm B}$ is $\mathcal{H}^n$-equivalent to some relatively closed minimizer for $\ell$ \cite[Theorem 6.1]{MNR1}. 
\end{itemize}

\noindent 

\subsection{Soap film capillarity model from \cite{KMS1,KMS2,KMS3}} 
Building upon \cite{MSS}, in which soap films were studied as regions of small volume, for a compact set $\wire\subset\mathbb{R}^{n+1}$, in \cite{KMS1} King-Maggi-Stuvard formulated the problem:
\begin{align}\label{soap film capillarity model old}
    \psi(v) &= \inf\{\mathcal{H}^n(\partial E \cap \Omega): \mbox{$E\in \mathcal{E}$, $|E|=v$, $\Omega \cap\partial E$ is $\mathcal{C}$-spanning $\wire$} \}\,,
\end{align}
where
\begin{align}\notag
    \mathcal{E}=\{E \subset \Omega : \mbox{ $E$ is open and $\partial E$ is $\mathcal{H}^n$-rectifiable} \}\,.
\end{align}
As explained in the introduction, one cannot expect a minimizing sequence in \eqref{soap film capillarity model old} to converge to an admissible competitor for \eqref{soap film capillarity model old}, so they also consider the relaxed class
\begin{align}\notag
   \mathcal{K}:= \{(K,E): \mbox{ $E\subset \Omega$}&\mbox{ is open with } \Omega \cap \cl(\partial^* E) = \Omega \cap \partial E\subset K\\ \label{k def}
   &\mbox{$K\subset \Omega$ is relatively closed and $\mathcal{H}^n$-rectifiable}\}\,.
\end{align}
with the corresponding energy
\begin{align}\notag
    \mathcal{F}_{\rm bd}(K,E) = \mathcal{H}^n( \Omega \cap \partial^* E ) + 2 \,\mathcal{H}^n (K \setminus \partial ^* E)\,;
\end{align}
see below for an explanation of the subscript ``bd". The main results of \cite{KMS1} are two-fold.
\begin{itemize}
    \item {\it Existence}: There exist {\bf generalized minimizers}, that is, there exists $(K,E)\in \mathcal{K}$ such that $K$ is $\mathcal{C}$-spanning $\wire$, $|E|=v$, and $\psi(v)=\mathcal{F}_{\rm bd}(K,E)$ \cite[Theorem 1.4]{KMS1}.
\item {\it Convergence to the Plateau problem}: $\psi(v)\to 2\,\ell$ with the corresponding subsequential convergence of minimizers in the sense of Radon measures \cite[Theorem 1.9]{KMS1}.
\end{itemize}
When $K \setminus \pa E \neq \varnothing$, the generalized minimizer $(K,E)$ is ``collapsed" and does not belong to the original admissible class from \eqref{soap film capillarity model old}. A natural conjecture therefore is that generalized minimizers for \eqref{soap film capillarity model old} are minimal among the relaxed class $\mathcal{K}$ of possibly collapsed soap films \cite[Remark 1.7]{KMS1}.

\subsection{Formulation and previous results from \cite{MNR1}}\label{sec:previous results}
Following \cite[Section 1.4]{MNR1}, we use Definition \ref{def homot span borel} to expand the soap film capillarity model to finite perimeter sets. We are actually able to treat two versions of the problem which differ in the choice of spanning set:
\begin{align}\label{soap film capillarity model}
    \psi_{\rm bd}(v) &= \inf\{P(E; \Omega): E\subset \Omega,\, P(E;\Omega)<\infty, \, |E|=v, \,\mbox{$\Omega \cap\partial^* E$ $\mathcal{C}$-spans $\wire$} \},\hspace{.5cm}\mbox{and} \\ \label{soap film capillarity model bulk}
    \psi_{\rm bk}(v) &= \inf\{P(E; \Omega):E\subset \Omega,\,P(E;\Omega)<\infty,\, |E|=v,\, E^\one \cup(\Omega \cap\partial^* E) \mbox{ $\mathcal{C}$-spans $\wire$} \}.
\end{align}
Here $P(E;\Omega)$, $|E|$, and $\partial^* E$ denote the relative perimeter of $E$ in $\Omega$ and the Lebesgue measure and reduced boundary of $E$, respectively. The subscripts ``bd" and ``bk" stand for boundary and bulk to emphasize the particular spanning set. The relaxed class is
\begin{align}\label{relaxed class}
    \kkms_{\rm B}&:= \big\{(K,E) :\mbox{$K,E\subset \Omega$ are Borel, $P(E;F)<\infty$ $\forall F \cc \Om$, and $\Omega \cap \partial^* E \shn K$}\big\}\,,
\end{align}
and the corresponding energies and variational problems are
\begin{align}\notag
    \mathcal{F}_{\rm bd}(K,E) &= \mathcal{H}^n( \Omega \cap \partial^* E ) + 2 \,\mathcal{H}^n (K \setminus \partial ^* E)\,, \\ \notag
    \mathcal{F}_{\rm bk}(K,E) &= \mathcal{H}^n( \Omega \cap \partial^* E ) + 2 \,\mathcal{H}^n (K \cap E^\zero)\,,
\end{align}
and
\begin{align}\label{modified kms relax prob}
    \Psi_{\rm bd}(v) &= \inf \{\mathcal{F}_{\rm bd}(K,E) : (K,E)\in \mathcal{K}_{\rm B}, \, |E|=v,\, \mbox{$K$ is $\mathcal{C}$-spanning $\wire$} \}\,\\ \label{modified kms relax bulk}
    \Psi_{\rm bk}(v) &= \inf \{\mathcal{F}_{\rm bk}(K,E) : (K,E)\in \mathcal{K}_{\rm B}, \, |E|=v,\, \mbox{$E^\one \cup K$ is $\mathcal{C}$-spanning $\wire$} \}\,.
\end{align}
Each problem \eqref{soap film capillarity model} and \eqref{soap film capillarity model bulk} has some advantages: for example the lower-dimensional spanning set $\Om \cap \pa^* E$ in \eqref{soap film capillarity model} is closer mathematically to the spanning surfaces in the Plateau problem $\ell$, while the bulk spanning problem \eqref{soap film capillarity model bulk} arises in an asymptotic limit of Allen-Cahn problems \cite{MNR2} with spanning constraint. In any case, we suspect that (generalized) minimizers for both problems coincide, at least for small volumes.
For the bulk problem $\Psi_{\rm bk}(v)$, in \cite{MNR1}, F. Maggi, D. Restrepo, and the author proved:
\begin{itemize}
    \item {\it Existence}: There exists $(K,E)\in \mathcal{K}\subset\mathcal{K}_{\rm B}$ such that $(K,E)$ is minimal for $\Psi_{\rm bk}(v)$.
    \item {\it Regularity}: There exists $\Sigma$ of Hausdorff dimension at most $n-7$ such that $(\Om \cap \pa^* E)\setminus \Sigma$ is smooth with constant mean curvature and $K \setminus (\pa E\cup \Sigma)$ is smooth with zero mean curvature. Also, $\Gamma=(\Om \cap \pa E) \setminus (\pa^* E\cup \Sigma)$ is locally $\mathcal{H}^{n-1}$-rectifiable, and, for any $x\in \Gamma$, there is $r>0$ such that $K\cap B_r(x)$ is a union of two $C^{1,1}$ hypersurfaces touching tangentially at $x$.
    \item {\it Convergence to the Plateau problem}: $\Psi_{\rm bk}(v)\to 2\,\ell$ with the corresponding subsequential convergence of minimizers in the sense of Radon measures.
\end{itemize}





\noindent Note that any admissible $E$ for $\psi_{\rm bd}(v)$ or $\psi_{\rm bk}(v)$ corresponds to a pair $(\partial^*E \cap \Omega,E)\in \mathcal{K}_{\rm B}$ which is admissible for $\Psi_{\rm bd}(v)$ or $\Psi_{\rm bk}(v)$, respectively. Therefore $\Psi_{\rm bd}(v) \leq \psi_{\rm bd}(v)$, and similarly $\Psi_{\rm bk}(v) \leq \psi_{\rm bk}(v)$, but equality is unclear. This leads to the following question:

\medskip
\begin{center}
{\bf Are the soap film capillarity model and its relaxation equivalent minimization problems, i.e. does $\bm{\Psi}_{\rm {\bf bd}}\bm{(v)}=\bm{\psi}_{\rm {\bf bd}}\bm{(v)}$ and $\bm{\Psi}_{\rm {\bf bk}}\bm{(v)}=\bm{\psi}_{\rm {\bf bk}}\bm{(v)}$?}
\end{center}

\subsection{Main results}\label{sec:results and discussion} In our main results, we answer this question. We obtain the equivalence as a corollary of an approximation theorem which we state first. 
\begin{theorem}[Approximation of Collapsed Competitors]\label{thm main of approximation}
    Let $\wire\subset \mathbb{R}^{n+1}$ be closed and $\mathcal{C}$ be a spanning class for $\wire$ such that $\ell<\infty$. If $(K,E)\in \kkms_{\rm B}$, $\mathcal{H}^n(K)<\infty$, and $\delta>0$, then there exists a set $E'\subset \Omega$ of finite perimeter in $\Omega$ such that $|E|=|E'|$,
\begin{align}\label{approximation theorem L1 closeness}
    |E' \Delta E| &\leq \delta\,, \\
\label{approximation theorem equation}
        \mathcal{H}^n(\partial^* E' \cap \Omega) &\leq \mathcal{F}(K,E) + \delta\,,
    \end{align}
and, if $K$ is $\mathcal{C}$-spanning $\wire$, so is $\Om \cap \pa^* E'$, while if $K \cup E^\one$ is $\mathcal{C}$-spanning $\wire$, so is $(E')^\one \cup (\Om \cap \pa^* E') $.
\end{theorem}
\noindent A consequence is that collapsed and non-collapsed formulations of the soap film capillarity problem are equivalent, and that limits of minimizing sequences for $\psi_{\rm bd}(v)$ or $\psi_{\rm bk}(v)$ are, up to resolving possible volume loss at infinity, minimizers for $\Psi_{\rm bd}(v)$ or $\Psi_{\rm bk}(v)$, respectively.

\begin{theorem}[Equivalence of collapsed and non-collapsed problems]\label{thm main of minimality}
   If $\wire\subset\mathbb{R}^{n+1}$ is closed and $\mathcal{C}$ is a spanning class for $\wire$ with $\ell<\infty$, then $\Psi_{\rm bd}(v)=\psi_{\rm bd}(v)$ and $\Psi_{\rm bk}(v)=\psi_{\rm bk}(v)$. Furthermore, if in addition $\wire$ is compact and $\{E_j\}_j$ is a minimizing sequence of finite perimeter sets for $\psi_{\rm bd}(v)$ or $\psi_{\rm bk}(v)$, then up to a subsequence, there exists $(K,E)\in \mathcal{K}_{\rm B}$ such that $E_j\to E$ locally in $\mathbb{R}^{n+1}$, 
\begin{align}\notag
    \mathcal{H}^n\mres(\Omega \cap \partial^* E_j) &\weakstar \mathcal{H}^n\mres(\Omega \cap \partial^* E) + 2\mathcal{H}^n\mres (K \setminus \partial^* E)\quad\mbox{or} \\ \notag
    \mathcal{H}^n\mres(\Omega \cap \partial^* E_j) &\weakstar \mathcal{H}^n\mres(\Omega \cap \partial^* E) + 2\mathcal{H}^n\mres (K \cap E^\zero),
\end{align}
and there exists $B_r(x)\subset(K \cup E \cup \wire)^c$ such that $|B_r(x)| = v - |E|$ and $(K \cup \pa B_r(x),E \cup B_r(x))$ is minimal for $\Psi_{\rm bd}(v)$ or $\Psi_{\rm bk}(v)$, respectively.
\end{theorem}
\begin{remark}[Loss of volume at infinity]
The presence of the ball $B_r(x)$ accounts for possible volume loss at infinity for a minimizing sequence. In the bulk problem $\Psi_{\rm bk}(v)$, this cannot happen for a minimizing sequence \cite[Appendix B]{MNR1}, and we expect the same to hold for $\Psi_{\rm bd}(v)$.
\end{remark}

One benefit of the complete variational theory for $\psi_{\rm bd}(v)$ and $\psi_{\rm bk}(v)$ provided by Theorem \ref{thm main of minimality} is in the study wetting and collapsing, which are expected whenever the minimizer for Plateau's problem $\ell$ admits a minimizer with Plateau-type singularities (as recently shown to be optimal in a capillarity model for planar soap bubbles \cite[Theorem 1.6]{MN23}; see also \cite{HM,BrakkeMorgan,Brakke05}). In particular, Theorem \ref{thm main of minimality} shows that the limiting pair $(K,E)$ for a minimizing sequence $\{E_j\}_j$ of $\psi_{\rm bd}(v)$ or $\psi_{\rm bk}(v)$ is a proper minimizer in the relaxed problem $\Psi_{\rm bd}(v)$ or $\Psi_{\rm bk}(v)$, respectively. As a consequence, the regularity statements proved for the bulk problem $\Psi_{\rm bk}(v)$ in \cite{MNR1} automatically hold for $(K,E)$ when $\{E_j\}_j$ are admissible for $\psi_{\rm bk}(v)$. In the boundary spanning problem, Theorem \ref{thm main of approximation} provides an alternative to the approximation arguments used in \cite{KMS2,KMS3} to study the Lagrange multipliers and singular sets of generalized minimizers in the presence of collapsing.

\subsection{Acknowledgments.} The author is grateful to Francesco Maggi and Daniel Restrepo for many fruitful discussions on these problems {\color{black}and to the editor and anonymous referees for their careful reading and comments}. 

\section{Notation and Preliminaries}\label{sec:prelim}

\subsection{Notation} As mentioned in the introduction, we will use the notation $A\overset{\mathcal{H}^k}{\subset} B$ to signify that $\mathcal{H}^k(A \setminus B)=0$. Similarly, $A \overset{\mathcal{H}^k}{=} B$ means that $\mathcal{H}^k(A \Delta B)=0$. For a Radon measure $\mu$, the $k$-dimensional lower density of $\mu$ at a point $x$, denoted by $\theta_*^k(\mu)(x)$, is computed via
\begin{align}\notag
    \theta_*^k(\mu)(x) = \liminf_{r\to 0^+}\frac{\mu(\cl B_r(x))}{\omega_k r^k}\,,
\end{align}
where $\omega_k$ is the $k$-dimensional volume of the unit ball in $\mathbb{R}^k$. A Borel set $S\subset \mathbb{R}^{n+1}$ is locally $\mathcal{H}^k$-rectifiable if it can be covered up to an $\mathcal{H}^k$-null set by a countable union of Lipschitz images of maps from $\mathbb{R}^k$ to $\mathbb{R}^{n+1}$. Lastly, for a locally $\mathcal{H}^k$-finite set $S$, we use the notation $\mathcal{R}(S)$ to denote the (locally) $\mathcal{H}^k$-rectifiable part of $S$, where, following \cite[13.1]{Simon}, any locally $\mathcal{H}^k$-finite set $S$ can be uniquely partitioned up to $\mathcal{H}^k$-null sets into a locally $\mathcal{H}^k$-rectifiable portion $\mathcal{R}(S)$ and a purely $\mathcal{H}^k$-unrectifiable portion $\mathcal{P}(S)$. The definition and properties of $\mathcal{P}(S)$ will not be relevant here.

\subsection{Preliminaries}
In this section we quote several results from \cite{MNR1}. First we have a fact showing that the rectifiable part of a $\mathcal{H}^n$-finite $\mathcal{C}$-spanning set is itself $\mathcal{C}$-spanning and the compactness theorem mentioned in the introduction. 

\begin{proposition}\cite[Lemma 2.2]{MNR1}\label{rect part is spanning prop}
    If $\wire\subset \mathbb{R}^{n+1}$ is closed, $\mathcal{C}$ is a spanning class for $\wire$, $S$ is $
    \mathcal{C}$-spanning $\wire$, and $\mathcal{H}^n\mres S$ is a Radon measure in $\Omega$, then $\mathcal{R}(S)$ is $\mathcal{C}$-spanning $\wire$. Moreover, the sets $T_1$ and and $T_2$ appearing in Definition \ref{def homot span borel} are sets of finite perimeter.
\end{proposition}


\begin{theorem}\cite[Theorem 1.4]{MNR1}\label{theorem collapsed capillarity compactness} Let $\wire$ be a closed set in $\mathbb{R}^{n+1}$, $\C$ be a spanning class for $\wire$, and $\{(K_j,E_j)\}_j$ be a sequence in $\K_{\rm B}$ such that $\sup_j\,\H^n(K_j)<\infty$, and let a Borel set $E$ and Radon measures $\mu_{\rm bk}$ and 
$\mu_{\rm bd}$ in $\Om$ be such that $E_j\to E$ locally in $L^1(\Om)$ and
\begin{eqnarray}
\label{def of mub}
&&\H^n\mres (\Om\cap\pa^*E_j) + 2\,\H^n \mres (\mathcal{R}(K_j) \cap E_j^\zero) \weakstar \mu_{\rm bk}\,,
\\
\label{def of mu}
&&\H^n\mres (\Om\cap\pa^*E_j) + 2\,\H^n \mres (\mathcal{R}(K_j) \setminus \pa^* E_j) \weakstar \mu_{\rm bd}\,,
\end{eqnarray}
as $j\to\infty$. Then the sets
\begin{eqnarray}
  \label{def of Kb limit}
  K_{\rm bk}\!\! &:=&\!\!\big(\Om\cap\partial^* E\big) \cup \Big\{x\in \Omega \cap E^\zero : \theta^n_*(\mu_{\rm bk})(x)\geq 2 \Big\}\,,
  \\
  \label{def of K limit}
  K_{\rm bd} \!\!&:=&\!\!\big(\Om\cap\partial^* E\big) \cup \Big\{x\in \Omega \setminus\pa^*E : \theta^n_*(\mu_{\rm bd})(x)\geq 2 \Big\}\,,
\end{eqnarray}
are such that $(K_{\rm bk},E),(K_{\rm bd},E)\in\K_{\rm B}$ and
\begin{eqnarray}
  \label{mub limit lb}
  \mu_{\rm bk}\!\!&\ge&\!\! \H^n\mres (\Om\cap\pa^*E) + 2\,\H^n \mres (K_{\rm bk} \cap E^\zero)\,,
  \\
  \label{mu limit lb}
  \mu_{\rm bd}\!\!&\ge&\!\!\H^n\mres (\Om\cap\pa^*E) + 2\,\H^n \mres (K_{\rm bd} \setminus\pa^*E)\,,
\end{eqnarray}
with
\begin{equation}
  \label{lsc Fb}
  \liminf_{j\to\infty}\F_{\rm bk}(K_j,E_j)\ge\F_{\rm bk}(K_{\rm bk},E)\,,\qquad
    \liminf_{j\to\infty}\F_{\rm bd}(K_j,E_j)\ge\F_{\rm bd}(K_{\rm bd},E)\,.
\end{equation}
Finally, if $K_j\cup E_j^\one$ (resp. $K_j$) is $\C$-spanning $\wire$ for every $j$, then $K_{\rm bk}\cup E^\one$ (resp. $K_{\rm bd}$) is $\C$-spanning $\wire$.
\end{theorem}

In addition to Theorem \ref{theorem collapsed capillarity compactness}, we will also use some further tools originating from its proof, for which we will need some terminology. If $U\subset\R^{n+1}$ is Lebesgue measurable, $\{U_i\}_i$ is a {\bf Lebesgue-partition of $U$} if $\{U_i\}_i$ is a finite or countable Borel partition of $U\setminus U'$ where $|U'|=0$.
Two Lebesgue-partitions $\{U_i\}_i$ and $\{\tilde{U}_j\}_j$ of $U$ are {\bf Lebesgue-equivalent} if there is a bijection $\s$ such that $|U_i\Delta \tilde{U}_{\s(i)}|=0$ for every $i$. If $S$ is a Borel set, a Lebesgue partition $\{U_i\}_i$ of $U$ {\bf induced by $S$} must satisfy
\begin{equation}
  \label{induced partition}
  \mbox{${\color{black}U^\one}\cap\pa^eU_i\shn S$}\qquad\forall i\,.
\end{equation}
Lastly, $\{U_i\}_i$ is called an {\bf essential partition of $U$ induced by $S$} if it is a Lebesgue partition of $U$ induced by $S$ and, for every $i$, $S$ does not essentially disconnect $U_i$.

\begin{theorem}\cite[Theorem 2.1]{MNR1}\label{theorem decomposition}
    If $U$ is {\color{black}a bounded set of finite perimeter}, and if $S$ is a Borel set with $\H^n(S \cap {\color{black}U^\one})<\infty$, then there exists an essential partition $\{U_i\}_{i\in I}$ of $U$ induced by $S$ such that each $U_i$ is a set of finite perimeter and
    \begin{equation}
      \label{essential partition perimeter bound}
          \sum_{i\in I}P(U_i;{\color{black}U^\one}) \leq 2\,\H^n(S \cap {\color{black}U^\one})\,.
    \end{equation}
    Moreover: 
    {\bf (a):} if $S^*$ is a Borel set with $\H^n(S^* \cap {\color{black}U^\one})<\infty$, $S^*$ is $\H^n$-contained in $S$, $\{U_j^*\}_{j\in J}$ is a Lebesgue partition of $U$ induced by $S^*$, and $\{U_i\}_i$ is an essential partition of $U$ induced by $S$, then
    \begin{equation}
      \label{essential partitions are monotone}
      \bigcup_{j\in J}\pa^*U^*_j\shn\bigcup_{i\in I}\pa^*U_i\,;
    \end{equation}
    {\bf (b):} if $S$ and $S^*$ are $\H^n$-finite sets in ${\color{black}U^\one}$, and either $S^*=\RR(S)$ or $S^*$ is $\H^n$-equivalent to $S$, then $S$ and $S^*$ induce $\mathcal{L}^{n+1}$-equivalent essential partitions of $U$.
\end{theorem}




\begin{theorem}\cite[Theorem 3.1, Remark 3.2]{MNR1}\label{theorem spanning with partition}
If $\wire\subset\R^{n+1}$ is a closed set in $\R^{n+1}$, $\C$ is a spanning class for $\wire$, $K$ is a Borel set locally $\H^n$-finite in $\Om$, $E$ has locally finite perimeter in $\Om$, and $\Om\cap\pa^*E\shn K$, then the set $S=K\cup E^{\one}$ is $\C$-spanning $\wire$ if and only if, for every $(\gamma,\Phi, T)\in \mathcal{T}(\C)$ and $\H^1$-a.e. $s\in\SS^1$, denoting by $\{U_i\}$ the essential partition of $T$ induced by $\RR(K)\cup T[s]$,
\begin{eqnarray}\label{spanning and the S partition equation}
&&\mbox{$T[s]\cap E^\zero\shn\cup_i \partial^* U_i$}\,.
\end{eqnarray}
\end{theorem}

\begin{remark}\label{trivial partial E containment}
    We point out that for any $(K,E)\in \mathcal{K}_{\rm B}$ and $(\gamma,\Phi,T)\in \mathcal{T}(\mathcal{C})$ and $s\in \mathbb{S}^1$,
    \begin{align}\label{partial E equation}
    \mbox{$\partial^* E\cap T$ is $\mathcal{H}^n$ contained in $\cup_i \partial^* U_i$}\,,
    \end{align}
    where $\{U_i\}$ is the essential partition of $T$ induced by $\mathcal{R}(K) \cup T[s]$. Indeed, $\{E \cap T,T\setminus E\}$ is a Lebesgue partition of $T$ induced by $\mathcal{R}(K) \cup T[s]$ since $T \cap\mathcal{R}(K) \overset{\mathcal{H}^n}{\supset}T \cap \partial^* E$. By Theorem \ref{theorem decomposition}(a), we deduce \eqref{partial E equation}.
\end{remark}

\begin{remark}[Consequences of $\ell<\infty$]\label{no points remark}
If $\mathcal{C}$ is a spanning class for $\wire$ such that $\ell<\infty$, then no $\gamma\in \mathcal{C}$ is homotopic in $\Omega$ to a point - if $\gamma\in \mathcal{C}$ were homotopic to a point, then the only $\mathcal{C}$-spanning set is $\Omega$, which has infinite $\mathcal{H}^n$-measure.
\end{remark}


\section{Proofs of the Approximation and Minimality Theorems}\label{sec:proofs}

\subsection{Proof of Theorem \ref{thm main of approximation}} The proof in the boundary spanning case proceeds in several steps and is based on an iteration procedure. The most delicate part of the iteration is ensuring that the reduced boundary of our approximating set will be $\mathcal{C}$-spanning. Each stage involves ``replacing" a portion of the $\mathcal{H}^n$-rectifiable set $\mathcal{R}(K)\setminus \partial^* E$ with a small set of finite perimeter $\mathcal{D}_i$. By choosing each $\D_i$ carefully to satisfy certain properties, we will be able to ensure that $\partial^* E\cup \cup_i \partial^* \D_i$ is $\mathcal{C}$-spanning and that
\begin{align}\notag \color{black}
    \partial^* \big[E \Delta \bigcup_i \D_i \big] \ehn \partial^* E \cup \bigcup_i \partial^* \D_i\,.
\end{align}
Together, these two properties will allow us to conclude that $\Om\cap\partial^* (E \Delta (\cup_i \D_i))$ is $\mathcal{C}$-spanning, and from here we fix the volumes to obtain $E'$. The bulk spanning case is significantly simpler, since we only have to ensure that $(E')^\one \cup(\Om \cap \pa^* E')$ is $\mathcal{C}$-spanning and not the (much) smaller set $\Om \cap \pa^* E'$.

\begin{proof}[Proof of Theorem \ref{thm main of approximation} for $\Psi_{\rm bd}$]
An outline is as follows. After making a preliminary reduction of the problem, we prove two necessary facts about measure-theoretic spanning. After that, in step two, we record a basic property of locally $\mathcal{H}^n$-rectifiable sets. Then, we devise a technical tool for the iteration in step three. We perform {\color{black}the iteration} in step four, and use Theorem \ref{theorem collapsed capillarity compactness} in steps five and six to take a limit and arrive at our approximating set $E'$.
\par
\medskip
\noindent\textit{Step zero}: Before making the approximation, we simplify the problem. We claim that to prove the theorem, it is enough to choose $\delta_j \to 0$ and produce a sequence of sets $\{E_j\}$ with $\mathcal{C}$-spanning reduced boundaries such that 
\begin{align}\label{need to show}
|E_j\Delta E|\leq \delta_j\quad \textup{and}\quad P(E_j;\Omega) \leq \mathcal{F}_{\rm bd}(K,E) + \delta_j\,.
\end{align}
To see why this is sufficient, by the $L^1$-convergence to $E$ and the volume-fixing variations lemma \cite[Section 29.6]{maggiBOOK}, there exists $C_0>0$, $C_1>0$ such that for all large $j$, we may obtain $G_j=g_j(E_j)$ for some diffeomorphism $g_j$ which is identity away from finitely many balls, and with
\begin{align}\notag
    |G_j|=|E|\,,\quad |G_j \Delta E_j| \leq C_0\Big||E|-|E_j| \Big|\,,\quad\textup{and}\quad |P(G_j;\Omega)-P(E_j;\Omega)| \leq C_1\Big||E|-|E_j| \Big|\,.
\end{align}
Given some $\delta>0$, by choosing $j$ large enough so that $\max\{C_0,C_1\}\delta_j<\delta$, $G_j$ satisfies the conclusions of Theorem \ref{thm main of approximation}. In particular, the $\mathcal{C}$-spanning requirement is fulfilled since $\partial^* E_j$ is $\mathcal{C}$-spanning and the diffeomorphisms $g_j$ preserve the property of being $\mathcal{C}$-spanning (as can be seen from Definition \ref{def homot span borel} and the fact that, if $B$ is Borel, $g_j(B^{\small(t)}) = (g_j(B))^{\small(t)}$ if $t=0,1$ and $g_j(\pa^e B) = \pa^e (g_j(B))$). Therefore, for the rest of the proof, we consider fixed small $\delta:=\delta_j$ and construct $E_j$ with $\mathcal{C}$-spanning reduced boundary such that \eqref{need to show} holds.
\par
\medskip
\noindent\textit{Step one}: Here we prove two claims. The second will ensure that our approximating sets are $\mathcal{C}$-spanning.
\par
\medskip
\noindent\textbf{Claim one:} If $(\gamma,\Phi,T)\in \mathcal{T}(\mathcal{C})$, $x\in T[s_0]$, $B_r(x)\setminus T[s_0]$ consists of two connected components $B_r^+(x)$ and $B_r^-(x)$ with $\pa B_r^+(x) \cap \pa B_r^-(x)=T[s_0]\cap B_r(x)$, and there exists an open set $D$ and ball $B_s(y)$ with $B_r(x)\subset D\subset B_s(y)\subset \Omega$, then $B_r^+(x)$ and $B_r^-(x)$ are subsets of two distinct open, connected components $A^+$ and $A^-$ of $(D \cap T)\setminus T[s_0]$.
\par
\medskip
The argument relies on the same construction \cite[Proof of Theorem A.1]{MNR1} that is behind the equivalence of Definition A and Definition B among closed sets.  {\color{black}Crucial to the argument is the inclusion $D\subset B_s(y)$, which ensures that $D$ is contractible in $\Omega$.}
\par
To prove the claim, suppose for contradiction that $B_r^+(x)$ and $B_r^-(x)$ both were subsets of a single open, connected component $A$ of $(D \cap T) \setminus T[s_0]$. We are going to use this to construct a curve in $\mathcal{C}$ homotopic to a point, which will be a contradiction. Since open connected sets in Euclidean space are path connected, for some $x_+\in B_r^+(x)$ and $x_-\in B_r^-(x)$, we can find a smooth curve $\gamma_1\subset (D \cap T) \setminus T[s_0]$ with endpoints $x_+$ and $x_-$. By Sard's theorem \cite[Proof of Lemma 10, Step 2]{DGM}, $\gamma_1$ can be deformed so that it meets $\pa B_{r}(x)$ transversally at a finite set of points, in which case $\gamma_1([0,1]) \setminus \cl B_{r}(x)$ consists of a finite disjoint union of arcs $\gamma_1((a_i,b_i))$ with $(a_i,b_i)\subset [0,1]$. Since $\gamma_1\cap T[s_0]=\varnothing$ and $T[s_0]$ disconnects $B_r(x)$, we may find $i$ such that $\gamma_1(a_i)\in \cl B_r^+(x) \cap \pa B_{r}(x)$ and $\gamma_1(b_i)\in \cl B_r^-(x) \cap \pa B_{r}(x)$ (modulo reversing the orientation of $\gamma_1$). Set $\gamma_2$ to be $\gamma_1$ restricted to $[a_i,b_i]$ so that $\gamma_2([a_i,b_i])\cap B_r(x)=\varnothing$. Now let $\g_3:[0,1]\to B_r(x)$ be an embedding with $\g_3(0)=\gamma_2(a_i)$, $\g_3(1)=\gamma_2(b_i)$ and such that $\g_2([0,1])\cap T[s]\cap B_{r(x)}(x)$ is one point $x_{3}=\g_3(t)$ and
\begin{align}\label{zero derivative}
\g_3'(t)\neq 0\,.
\end{align}
We may further arrange $\gamma_3$ so that the concatenation $\gamma_*$ of $\gamma_2$ and $\gamma_3$ is smooth. Letting $\mathbf{p}_{\mathbb{S}^1}$ denote the canonical projection map from $\mathbb{S}^1\times B_1^n$ to $\mathbb{S}^1$, we find from \eqref{zero derivative} and $\gamma_2\cap T[s_0]=\varnothing$ that the $\mathbb{S}^1$-valued curve $\mathbf{p}_{\mathbb{S}^1}(\Phi^{-1}(\gamma_*))$ has Brouwer degree either $+1$ or $-1$ \cite[pg. 27]{Mil97}. If its degree is $-1$, let us set $\gamma_{**}(s)=\gamma_{*}(\overline{s})$ (where the bar denotes complex conjugation) so that $\gamma_{**}$ has the same image as $\gamma_*$ and the degree of $\mathbf{p}_{\mathbb{S}^1}(\Phi^{-1}(\gamma_{**}))=+1$. If the degree is $+1$, simply set $\gamma_{**}=\gamma_*$. Now the $\mathbb{S}^1$-valued curves $\mathbf{p}_{\mathbb{S}^1}(\Phi^{-1}(\gamma_{**}))$ and $\mathbf{p}_{\mathbb{S}^1}(\Phi^{-1}(\gamma))$ both have winding number 1 and are therefore homotopic to one another. Since $B_1^n$ is convex and $\Phi$ is a diffeomorphism, then $\gamma$ is homotopic to $\gamma_{**}$. {\color{black}But this is impossible, since $\gamma_{**} \subset D \subset B_s(y)\subset \Omega$, implying the contractibility of $D$ in $\Omega$, and, by $\ell<\infty$, $\mathcal{C}$ does not contain any curves homotopic in $\Omega$ to a point (see Remark \eqref{no points remark}).} This gives the desired contradiction and thus concludes the proof of claim one.
\par
\medskip
\noindent\textbf{Claim two:} If $(\gamma,\Phi,T)\in \mathcal{T}(\mathcal{C})$ and $D$ is an open bounded set of finite perimeter with $D\subset B_s(y)\subset \Omega$ for some ball $B_s(y)$, then for $\mathcal{H}^n$-a.e. $x\in T[s_0] \cap D$, there exists an element $V_i$ of the essential partition of $T$ induced by $\partial^* D \cup T[s_0]$ such that $x\in \partial^* V_i$.
\par
\medskip
To see why this second claim follows from the first, fix any $B_r(x) \subset D \cap T$ centered at $x\in T[s_0] \cap D$ small enough so that $T[s_0]$ divides $B_r(x)$ into two open connected components, $B_r^+(x)$ and $B_r^-(x)$ with common boundary $T[s_0] \cap B_r(x)$. By claim one, $B_r^+(x)$ and $B_r^-(x)$ are subsets of distinct connected components $A_+$ and $A_-$ of $(D\cap T)\setminus T[s_0]$. Moreover, $A_+$ and $A_-$ are sets of finite perimeter by \cite[4.5.11]{F} since $\partial^e A_{\pm} \subset T[s_0]\cup \partial^e D \cup \partial T$. Then setting $B = (D \cap T)\setminus (A_+ \cup A_-)$ and $C= T \setminus D$ (note that $C\neq \varnothing$ since if it were we would have $T\subset B_s(y)$ so that $\gamma$ is homotopic to a point), we find that $\{A_+,A_-,B,C\}$ is a non-trivial Lebesgue partition of $T$ into sets of finite perimeter such that, by standard facts about unions/intersections of sets of finite perimeter \cite[Ch. 16]{maggiBOOK},
\begin{align}\label{thingy}
  T \cap ( \partial^* A_+ \cup \partial^* A_- \cup \partial^* B \cup \partial^* C) \mbox{ is $\mathcal{H}^n$-contained in $\partial^* D \cap T[s_0]$}\,.
\end{align}
In words, it is a Lebesgue partition of $T$ induced by $\partial^* D \cap T[s_0]$, and so therefore Theorem \ref{theorem decomposition}(a) implies that $\partial^* A_+ \cap T $ is $\mathcal{H}^n$-contained in $\cup_i \partial^* V_i$, where $\{V_i\}_i$ is the essential partition of $T$ induced by $\partial^* D \cap T[s_0]$. Since $T[s_0] \cap B_r(x) \subset \partial^* A_+ \cap T$ and $x\in T[s_0] \cap D$ was arbitrary, this finishes the claim.
\par
\medskip
\noindent\textit{Step two}: Here we show that if $R\subset \Omega$ is a locally $\mathcal{H}^n$-rectifiable set with finite $\mathcal{H}^n$-measure, then up to an $\mathcal{H}^n$-null set, it can be decomposed into a countable union of pairwise disjoint compact Lipschitz graphs, each of which is contained in some ball disjoint from $\wire$. More precisely, we can write
\begin{align}\label{kin and kout def}
    R = A_0 \cup \bigcup_{m=1}^\infty f_m(A_m)
\end{align}
where $\mathcal{H}^n(A_0)=0$, and, for each $m\geq 1$, there exist $v_m\in \mathbb{S}^1$, Lipschitz function $F_m:v_m^\perp \to \mathbb{R}$, and compact $A_m \subset v_m^\perp$ such that
$$
f_m(y) =y+ F(y)v_m \quad \forall y\in A_m\,.
$$
and, for some $B_{t_m}(x_m) \subset\!\subset \Omega$,
\begin{align}\label{AM ball containment}
    f_m(A_m) \subset B_{t_m}(x_m)\,.
\end{align}
To see this, recall that any locally $\mathcal{H}^n$-rectifiable set can be covered by a family of Lipschitz graphs \cite[Proposition 2.76]{AFP} and can therefore be decomposed up to an $\mathcal{H}^n$-null set $B_0$ into a countable union of pairwise disjoint Lipschitz graphs $f_m(B_m)$, where each $B_m$ is a Borel subset of some $n$-dimensional plane. Since each such $B_m$ is $\mathcal{H}^n$-equivalent to a countable union of disjoint compact sets, which we make take to be small enough so that their images under $f_m$ are each contained in some ball disjoint from $\wire$, the desired decomposition follows. In the next step, we will use the notation $\partial_{v_m^\perp}$ to denote the natural notion of boundary for subsets of $v_m^\perp$ that arises by identifying $v_m^\perp$ with $\mathbb{R}^n$. Let also denote by $\mathbf{p}_{v_m^\perp}$ the projection onto $v_m^\perp$.
\par
\medskip
\noindent\textit{Step three}: In this step we prove the following claim, which is the main technical tool for the iteration.
\par
\medskip
\noindent{\bf Claim:} If $F\subset \Omega$ is relatively closed, $E \subset \Omega$ is a set of locally finite perimeter in $\Omega$, $R\subset \Omega \setminus F$ is a locally $\mathcal{H}^n$-rectifiable set with finite $\mathcal{H}^n$-measure, $\{v_m\}$, $\{f_m\}$, $\{A_m\}$, and $\{B_{t_m}(x_m)\}$ are vectors, functions, compact sets, and balls corresponding to the decomposition of $R$ from step two, and $\beta\in (0,1)$, then there exists $M\in \mathbb{N}$ and open sets of finite perimeter $D_1, \dots, D_M$ with the following properties:
\begin{align}\label{eats up R}
\mathcal{H}^n\Big(R \setminus \bigcup_{m=1}^M \cl D_m \Big) &< \frac{\mathcal{H}^n(R)}{2}\,,
    \\ \label{no overlap with each other}
    \cl(D_m) \cap \cl(D_{m'})&=\varnothing\quad\mbox{for }m\neq m', 1\leq m, m' \leq M\,,  \\ \label{farness from D0 1}
\dist ( D_m ,  F)>&0\quad\forall m\leq M\,, \\ \label{okay for spanning}
    f_m(A_m)\subset D_m &\subset\!\subset B_{t_m}(x_m)\subset \Omega\quad \forall m\leq M\,,
    \\
 \label{rect of Cm Dm}
    \partial D_m\textup{ is Lipschitz }&\textup{ with $\mathcal{H}^n(\partial D_m \setminus \partial^* D_m)=0$}\quad\forall m\leq M\,, \\ \label{small volume}
    \sum_{m=1}^M|D_m| &\leq \beta\,,
 \\ \label{farness from D0 2}
 \Big|B_s(z) \cap \bigcup_{m=1}^M D_m \Big| &< \beta|B_s(z)|\quad \forall z\in  F \cap \Omega\,,\, s>0\,,\\    \label{no overlap with E}
    \mathcal{H}^n(\partial D_m &\cap \partial^* E )=0\quad \forall m\leq M\,, \\ 
 \label{closeness of Hausdorff business}
    P(D_m)&\leq 2\,\mathcal{H}^n(f_m(A_m)) + \frac{\beta}{2^m}\quad \forall m \leq M\,.
\end{align}
Since there are many requirements, to aid the reader we will make our construction in an order that mirrors the order of \eqref{eats up R}-\eqref{closeness of Hausdorff business}, refining as we go so as to satisfy each successive requirement.
\par
First, let us choose $M\in \mathbb{N}$ large enough so that
\begin{align}\notag
   \sum_{m=M+1}^\infty \mathcal{H}^n(f_m(A_m)) < \frac{\mathcal{H}^n(R)}{2}\,.
\end{align}
With $M$ chosen as such, any open sets $D_m$ satisfying $f_m(A_m)\subset D_m$ will therefore give
\begin{align}\label{verify 4.5}
    \mathcal{H}^n\Big(R \setminus \bigcup_{m=1}^M \cl D_m \Big)=\sum_{m'=1}^\infty\mathcal{H}^n\Big(f_{m'}(A_{m'}) \setminus \bigcup_{m=1}^M \cl D_m \Big) \leq \sum_{m=M+1}^\infty \mathcal{H}^n(f_m(A_m)) < \frac{\mathcal{H}^n(R)}{2}\,,
\end{align}
which is \eqref{eats up R}. Next, using the fact that $f_m(A_m)$ are pairwise disjoint compact sets each of which is at positive distance from $F$ (by virtue of $R\subset \Omega \setminus F$ and the relative closedness of $F$) and each other, we choose may choose
open sets $A_m\subset T_m\subset v_m^\perp$ with smooth boundary such that for $m\leq M$,
\begin{align}\label{pairwise disjoint}
\min\{ \dist( f_m(T_m), f_{m'}(T_{m'}))&,\,\dist(f_m(T_m),F)\}>0\quad\forall m'\neq m\,, 
\\ \label{compact containment 1}
    f_m(\cl T_m) &\overset{\eqref{AM ball containment}}{\subset} B_{t_m}(x_m)\,, \\ \label{close lebesgue measure}
    \mathcal{H}^n(f_m(T_m)) &< \mathcal{H}^n(f_m(A_m)) + \frac{\beta}{3\times 2^{m}}\,.
\end{align}
Let us set $(T_m)_t=\{x\in T_m: \dist(x,\partial_{v_m^\perp} T_m)>t\}$. Now, for each $m\in \{1,\dots,M\}$, by \eqref{pairwise disjoint}-\eqref{compact containment 1}, there exists $\tau_{m,0}>0$ such that for any $\tau_{m,i}<\tau_{m,0}$, $i=1,2$, 
\begin{align}\notag
    A_m \subset (T_m)_{\tau_{m,1}}
\end{align}
and the cylindrical type sets over the graphs $f_m((T_m)_{\tau_{m,1}})$ defined by
\begin{align*}
    C_{\tau_{m,1},\tau_{m,2}}^m:=\{z\in \mathbb{R}^{n+1}:\mathbf{p}_{v_m^\perp}(z)\in (T_m)_{\tau_{m,1}},\, |z\cdot v_m - F_m(\mathbf{p}_{v_m^\perp}(z))|<\tau_{m,2} \}\,,
\end{align*}
satisfy 
\begin{align}\label{pairwise disjoint 2}
\min\{ \dist(C_{\tau_{m,1},\tau_{m,2}}^m, C_{\tau_{m',1},\tau_{m',2}}^{m'})&,\,\dist(C_{\tau_{m,1},\tau_{m,2}}^m,F)\}>0\quad\forall m'\neq m, \tau_{m',i}<\tau_{m',0}\,, 
\\ \label{compact containment 2}
    f_m(\cl C_{\tau_{m,1},\tau_{m,2}}^m) &\subset B_{t_m}(x_m)\,,
\end{align}
that is, \eqref{no overlap with each other}-\eqref{okay for spanning} with $C_{\tau_{m,1},\tau_{m,2}}^m=D_m$. Postponing our choice of $\tau_{m,i}$ until later on (we will need this choice for \eqref{no overlap with E}), note that each $\partial C_{\tau_{m,1},\tau_{m,2}}^m$ is Lipschitz since it is the image through the Lipschitz map
$$
x\mapsto \mathbf{p}_{v_m^\perp}(x) + [F_m(\mathbf{p}_{v_m^\perp}(x))+x\cdot v_m]v_m
$$
of the Lipschitz cylinder $\partial_{v_m^\perp} (T_m)_{\tau_{m,1}} + \{tv_m:|t|\leq \tau_{m,2}\}$, which is \eqref{rect of Cm Dm}. Furthermore, it is clear that we may decrease the $\tau_{m,0}$'s if necessary to ensure \eqref{small volume}. For \eqref{farness from D0 2}, we first notice that
\begin{align}\label{cyl vol to 0}
  |C_{\tau_{m,1},\tau_{m,2}}^m| = 2\tau_{m,2} \mathcal{H}^n((T_m)_{\tau_{m,1}})\leq 2\tau_{m,0} \mathcal{H}^n(T_m)\to 0\quad\textup{as }\tau_{m,0}\to 0\,.  
\end{align}
Together with the fact that $C_{\tau_{m,1},\tau_{m,2}}^m$ is at positive distance from $F$ (by \eqref{pairwise disjoint 2}), \eqref{cyl vol to 0} guarantees that we can decrease $\tau_{m,0}$ if necessary to ensure that for any $\tau_{m,1},\tau_{m,2}<\tau_{m,0}$,
\begin{align}\label{farness from D0 3}
    \Big|B_s(z) \cap \bigcup_{m=1}^M C_{\tau_{m,1},\tau_{m,2}}^m \Big| &< \beta|B_s(z)|\quad \forall z\in F \cap \Omega\,,\, s>0\,,
\end{align}
which is \eqref{farness from D0 2} with $C_{\tau_{m,1},\tau_{m,2}}^m=D_m$. To recap, every choice of $\tau_{m,1},\tau_{m,2}<\tau_{m,0}$ and corresponding $C_{\tau_{m,1},\tau_{m,2}}^m=D_m$ gives sets that satisfy \eqref{eats up R}-\eqref{farness from D0 2}, and so we must choose the small parameters so as to satisfy \eqref{no overlap with E}-\eqref{closeness of Hausdorff business}. By the smoothness of $\partial_{v_m^\perp} T_m$, we may decrease each $\tau_{m,0}$ so that for $\tau<\tau_{m,0}$, $\partial_{v_m^\perp} (T_m)_{\tau}$ is smooth. Now, aiming towards \eqref{no overlap with E}, we recall that since $\mathcal{H}^n(\partial^* E )<\infty$, for any uncountable family of pairwise
disjoint sets, $\partial^*E$ can have $\mathcal{H}^n$-positive overlap with at mostly countably many. Therefore, there is an at most countable set $\mathcal{T}_1\subset (0,\min_m \tau_{m,0})$ such that
\begin{align}\label{4.23}
    \sup_{1\leq m\leq M} \big\{\mathcal{H}^n(\partial^* E \cap \{x\in \mathbb{R}^{n+1}: \mathbf{p}_{v_m^\perp}(x)\in \partial_{v_m^\perp}(T_m)_\tau\})\big\}=0 \quad \forall \tau\notin \mathcal{T}_1\,.
\end{align}
Let use choose $\tau_1\notin \mathcal{T}_1$ and set $\tau_{m,1}=\tau_1$ for $1\leq m \leq M$. Then by \eqref{4.23}, the lateral boundaries of our sets $C_{\tau_{m,1},\tau_{m,2}}^m$ will have trivial $\mathcal{H}^n$-overlap with $\partial^* E$. To now choose $\tau_{m,2}$, we compute
\begin{align}\notag
    \mathcal{H}^n(\partial C_{\tau_{m,1},\tau_{m,2}}^m)&=2\,\mathcal{H}^n(f_m((T_m)_{\tau_1}))+ 2\,\tau_{m,2} \mathcal{H}^{n-1}(\partial_{v_m^\perp} (T_m)_{\tau_1}) \\ \label{nhbd close perimeter tm}
    &\overset{\eqref{close lebesgue measure}}{<} 2\,\mathcal{H}^n(f_m(A_m)) + \frac{2\beta}{3\times 2^{m}} + 2\,\tau_{m,2} \mathcal{H}^{n-1}(\partial_{v_m^\perp} (T_m)_{\tau_1}) \,. 
\end{align}
Next, by the exact same countability argument as leading to \eqref{4.23}, there is an at most countable set $\mathcal{T}_2\subset (-\min_m \tau_{m,0},\min_m \tau_{m,0})$ such that 
\begin{align}\label{4.23 redux}
    \sup_{1\leq m\leq M} \big\{\mathcal{H}^n(\partial^* E \cap \{x\in \mathbb{R}^{n+1}: F(\mathbf{p}_{v_m^\perp}(x))+\tau= x\cdot v_m  \})\big\}=0 \quad \forall \tau\notin \mathcal{T}_2\,.
\end{align}
We choose $\tau_2$ with $\pm\tau_2\notin\mathcal{T}_2$ and small enough so that
\begin{align}\notag
    \sup_{1\leq m \leq M}2\,\tau_{2} \mathcal{H}^{n-1}(\partial_{v_m^\perp} (T_m)_{\tau_1})< \frac{\beta}{3\times 2^m}
\end{align}
and set $D_m = C_{\tau_{1},\tau_{2}}^m$. By \eqref{nhbd close perimeter tm} and the choice of $\tau_2$, we have
\begin{align}\notag
    \mathcal{H}^n(\partial D_m) < 2\,\mathcal{H}^n(f_m(A_m)) + \frac{\beta}{2^m}\,,
\end{align}
which is \eqref{closeness of Hausdorff business}. Lastly, by the definitions of $D_m$ and $\mathcal{T}_1$ and $\mathcal{T}_2$, we have 
\begin{align}\notag
    \mathcal{H}^n(\partial D_m\cap \partial^* E)= \mathcal{H}^n(\partial C_{\tau_{1},\tau_{2}}^m\cap \partial^* E) &\leq \mathcal{H}^n(\partial^* E \cap \{x\in \mathbb{R}^{n+1}: \mathbf{p}_{v_m^\perp}(x)\in \partial_{v_m^\perp}(T_m)_{\tau_1}\}) \\ \notag
    &\quad+ \mathcal{H}^n(\partial^* E \cap \{x\in \mathbb{R}^{n+1}: F(\mathbf{p}_{v_m^\perp}(x))\pm\tau_2= x\cdot v_m  \})\\ \notag
    &=0\,,
\end{align}
which is \eqref{no overlap with E}.
\par
\medskip
\noindent\textit{Step four}: For the rest of the proof, we fix $\delta_j$ and attempt to verify the reduction of the theorem laid out in step zero with this $\delta_j$. Here we will iteratively apply the previous step's claim to $\mathcal{R}(K)\setminus \partial^* E$. To begin with, we apply the claim with $F=\varnothing$, $R=\mathcal{R}(K)\setminus \partial^* E$, and $\beta=\min\{1/8,\delta_j/4\}$, yielding $M_0\in \mathbb{N}$ and open sets of finite perimeter $D_1^0,\dots, D_M^0$. If we set $\mathcal{D}_0=\cup_{m=1}^M D_m^0$, then by \eqref{no overlap with each other}, $\partial \mathcal{D}_0$ is the disjoint union of $\partial D_m^0$, and is therefore Lipschitz with $P(\mathcal{D}_0) = \sum_{m=1}^{M_0}P(D_m^0)$. By \eqref{eats up R}-\eqref{closeness of Hausdorff business}, our sets $D_m^0$ and $\mathcal{D}_0$ satisfy:
\begin{align}\label{eating up 0}
    \mathcal{H}^n \big( \mathcal{R}(K)\setminus (\partial^* E \cup \cl \mathcal{D}_0) \big) &< \frac{\mathcal{H}^n(\mathcal{R}(K)\setminus \partial^* E)}{2} \\ \label{ok for spanning D0}
    f_m^0(A_m^0) \subset D_m^0 \subset\!\subset &B_{t_{m,0}}(x_{m,0}) \subset \Omega \quad \forall m\leq M_0 \\ \label{small volume D0}
    |\mathcal{D}_0| &\leq \frac{\delta_j}{4}\,,  \\ \label{D0 no overlap with E}
    \mathcal{H}^n(\partial \mathcal{D}_0 &\cap \partial^* E) = 0 \,, \\ \label{per est D0}
    P(\mathcal{D}_0)=\sum_{m=1}^{M_0}P(D_m^0) \leq \sum_{m=1}^{M_0} 2\mathcal{H}^n(f_m^0(A_m))& + \frac{\delta_j}{4} \leq 2\mathcal{H}^n(\mathcal{R}(K) \cap \cl\mathcal{D}_0 \setminus \partial^* E) + \frac{\delta_j}{4}\,.
\end{align}
With the initial step complete, now we iteratively apply the claim for $k=1,2,3,\dots$, with, at the $k$-th stage, $F=F_{k-1} = \cl \mathcal{D}_{0}\cup \dots \cup \cl \mathcal{D}_{k-1}$, $R=R_{k} = \mathcal{R}(K) \setminus (\partial^* E \cup F_{k-1})$, and $\beta=\min\{2^{-k-3},\delta_j / 2^{k+2}\}$. We obtain a sequence of families of sets  $D_1^{k},\dots, D_{M_k}^{k}$ with $\mathcal{D}_k = \cup_{m=1}^{M_k}D_m^k$ having Lipschitz boundary $\cup_{m=1}^{M_k} \partial D_m^k$ and satisfying, for each $k\geq 1$,
\begin{align}\label{eating up k}
    \mathcal{H}^n \big( \mathcal{R}(K)\setminus (\partial^* E\cup F_{k-1} \cup \cl \mathcal{D}_{k} )\big) &< \frac{\mathcal{H}^n(\mathcal{R}(K)\setminus (\partial^* E\cup F_{k-1}))}{2}\,,
    \\ \label{away from Fk Dk}
    \dist( \mathcal{D}_k, F_{k-1}) > 0\,,
    \\ \label{ok for spanning Dk}
    f_m^k(A_m^k) \subset D_m^k \subset\!\subset &B_{t_{m,k}}(x_{m,k}) \subset \Omega \quad \forall m\leq M_k\,, \\ \label{small volume Dk}
    |\mathcal{D}_k| &\leq \frac{\delta_j}{2^{k+2}}\,,  \\ \label{boundary Dk estimate}
    |B_s(z) \cap \mathcal{D}_k| < \frac{|B_s(z)|}{2^{k+3}}&\quad \forall z\in \cl \mathcal{D}_0 \cup \dots \cup \cl \mathcal{D}_{k-1}\,,\, s>0\,, \\
    \label{Dk no overlap with E}
    \mathcal{H}^n(\partial \mathcal{D}_k &\cap \partial^* E) = 0 \,, \\ \label{per est Dk}
    P(\mathcal{D}_k)=\sum_{m=1}^{M_k}P(D_m^k) \leq \sum_{m=1}^{M_k} 2\mathcal{H}^n(f_m^k(A_m^k))& + \frac{\delta_j}{2^{k+2}} \leq 2\mathcal{H}^n(\mathcal{R}(K) \cap \cl\mathcal{D}_k \setminus \partial^* E) + \frac{\delta_j}{2^{k+2}}\,.
\end{align}
We remark that a consequence of \eqref{away from Fk Dk} and the fact that each $\mathcal{D}_k$ has Lipschitz boundary is
\begin{align}\label{dk breakdown}
   \mathcal{H}^n\mres (\partial^* F_k) = \sum_{k'=0}^k \mathcal{H}^n \mres (\partial^* \mathcal{D}_{k'})\,.
\end{align}
Now by repetitively using \eqref{eating up k} and finally \eqref{eating up 0}, we obtain
\begin{align}\label{final eating up}
    \mathcal{H}^n\big(\mathcal{R}(K) \setminus (\partial^* E \cup F_k ) \big)< \frac{\mathcal{H}^n(\mathcal{R}(K)\setminus\partial^* E)}{2^{k+1}}\,.
\end{align}
\par
We also claim that for each $k\geq 0$, the $\mathcal{H}^n$-rectifiable sets 
\begin{align}\notag
   S_k:= (\Omega \cap \partial^* E)\cup \partial F_k \cup  &[\mathcal{R}(K) \setminus (F_k\cup \partial^* E)]   \\ \label{hn equiv sets}
    &\quad \mbox{ are $\mathcal{H}^n$-equivalent to }[\Omega \cap \partial^*( E \Delta F_k)] \cup [\mathcal{R}(K) \setminus (F_k \cup \partial^* E )]
\end{align}
and are also $\mathcal{C}$-spanning. For the $\mathcal{H}^n$-equivalence, we first observe that since $\mathcal{D}_{k'}$ are each at mutual positive distance from each other for $0\leq k' \leq k$ and have Lipschitz boundaries,
\begin{align}\label{breakdown of partial F}
    \partial F_k \overset{\mathcal{H}^n}{=}\partial^* F_k \overset{\mathcal{H}^n}{=} \cup_{k'=0}^k \partial \mathcal{D}_{k'}\,.
\end{align}
Therefore, by the ``non-overlapping" \eqref{D0 no overlap with E} and \eqref{Dk no overlap with E} of each $\partial \mathcal{D}_{k'}$ with $\partial^* E$, we have
\begin{align}\label{Fk no overlap with E}
    \mathcal{H}^n(\partial^* E \cap \partial^* F_k) = 0\,.
\end{align}
As a consequence, the formula for the reduced boundary of the symmetric difference of two sets \cite[Exercise 16.5]{maggiBOOK} gives 
\begin{align}\label{application of maggi 16}
    \partial^*(E \Delta F_k)\overset{\mathcal{H}^n}{=}(\partial^* E \setminus \partial^* F_k) \cup (\partial^* F_k \setminus \partial^* E) \overset{\mathcal{H}^n}{=} \partial^* E \cup \partial^* F_k \overset{\mathcal{H}^n}{=} \partial^* E \cup \partial F_k\,,
\end{align}
which immediately implies \eqref{hn equiv sets}. Note that as a consequence of this equivalence,
\begin{align}\label{belongs to Kb}
    (S_k,E \Delta F_k) \in \mathcal{K}_{\rm B}\,,
\end{align}
with \eqref{application of maggi 16}, \eqref{hn equiv sets}, and the $\mathcal{H}^n$-equivalence of $\partial F_k$ and $\partial^* F_k$ giving
\begin{align}\label{computation of R}
    \mathcal{R}(S_k) \setminus \partial^* (E \Delta F_k) \overset{\mathcal{H}^n}{=} S_k \setminus (\partial^* E \cup \partial^* F_k) \overset{\mathcal{H}^n}{=} \mathcal{R}(K) \setminus (F_k \cup \partial^* E)\,.
\end{align}
To prove that $S_k$ is $\mathcal{C}$-spanning, let us fix $(\gamma,\Phi,T)\in \mathcal{T}(\mathcal{C})$. We first recall that $\mathcal{R}(K)$ is $\mathcal{C}$-spanning by Proposition \ref{rect part is spanning prop}. Therefore, by Theorem \ref{theorem spanning with partition}, there exists $J\subset \mathbb{S}^1$ of full $\mathcal{H}^1$-measure such that if $s\in J$, then, letting $\{U_i\}_i$ denote the essential partition of $T$ 
induced by $\mathcal{R}(K) \cup T[s]$,
\begin{eqnarray}\label{spanning equation step four}
&&\mbox{$T[s]$ is $\H^n$-contained in $\cup_i \partial^* U_i$}\,.
\end{eqnarray}
Again by Theorem \ref{theorem spanning with partition}, to show that $S_k$ is $\mathcal{C}$-spanning, it is enough to show that for $s\in J$, letting $\{V_i\}_i$ denote the essential partition of $T$ induced by $\mathcal{R}(S_k)\cup T[s]$, 
\begin{eqnarray}\label{spanning equation step four Sk}
&&\mbox{$T[s]$ is $\H^n$-contained in $\cup_i \partial^* V_i$}\,.
\end{eqnarray}
\par
Towards demonstrating \eqref{spanning equation step four Sk}, we first claim that
\begin{align}\notag
    \mathcal{P}_k:=\{W : W= U_i \setminus F_k \mbox{ or }W =T \cap D_m^{k'}\mbox{ for some $k' \leq k$ and $m\leq M_{k'}$}\}
\end{align}
is a Lebesgue partition of $T$ induced by $S_k\cup T[s]$. Since $\{U_i\}$ is a Lebesgue partition of $T$ and $F_k$ is the disjoint union of $D_m^{k'}$ for $k'\leq k$ and $m\leq M_{k'}$, the fact that this is a Lebesgue partition of $T$ follows. Let us enumerate its (nontrivial) sets as $W_i$. To see that $\mathcal{P}_k$ is induced by $S_k\cup T[s]$, we begin by appealing to \cite[Ch. 16]{maggiBOOK} and the fact that $\{U_i\}$ is induced by $\mathcal{R}(K) \cup T[s]$ to see that
\begin{align}\label{induced 1}
  T \cap   \partial^* (U_i \setminus F_k)\shn T \cap [(\partial^* U_i \cap F_k^{(0)})\cup \partial^* F_k ] \subset T \cap \big[\big((\mathcal{R}(K)\cup T[s]) \setminus F_k\big) \cup \partial F_k\big] \shn S_k \cup T[s]\,,
\end{align}
where the last inclusion follows by the definition of $S_k$. On the other hand, if $W_i=T \cap D_m^{k'}$, then
\begin{align}\label{partial Dk in Sk}
\partial^* W_i \cap T \subset \partial D_m^{k'}\cap T\subset \partial F_k \cap T\subset S_k\,.
\end{align}
Therefore, since $\mathcal{P}_k$ is induced by $S_k \cup T[s]$, we may employ Theorem \ref{theorem decomposition}(a) to conclude that
\begin{align}\label{wi containment}
    \cup_i \partial^* W_i \shn \cup_i \partial^* V_i\,.
\end{align}
Finally we check \eqref{spanning equation step four Sk} in three cases. By \eqref{spanning equation step four}, it is enough to verify among those $x\in T[s]$ such that $x\in \partial^* U_i$ for some $i$. In the first case, suppose that $x\in T[s]$ is such that $\dist(x,F_k)>0$ and $x\in \partial^* U_i$ for some $i$. Then $x\in \partial^* (U_i\setminus F_k)$, and thus $\mathcal{H}^n$-a.e. such point belongs to $S_k \cup T[s]$ as desired by \eqref{wi containment}. The second two cases both involve $x\in T[s]$ such that $\dist(x,F_k)=0$. Suppose that $x\in \partial F_k$. We recall from Remark \ref{trivial partial E containment}, which applies to the pair $(S_k,E\Delta F_k)$ by \eqref{belongs to Kb}, that $\partial^* (E \Delta F_k) \cap T$ is $\mathcal{H}^n$-contained in $\cup_i \partial^* V_i$. Since $\partial^* (E\Delta F_k)\overset{\mathcal{H}^n}{=}\partial^* E \cup \partial F_k$ by \eqref{application of maggi 16}, we find that $\mathcal{H}^n$-a.e. $x\in T[s]\cap \partial F_k$ belongs to some $\partial^* V_i$ as desired. Lastly, let us consider those $x\in \mathrm{int}\,F_k \cap T[s]$; it is for these $x$ that we will use the result of step one. By the definition of $F_k$, every such $x$ must belong to $D_m^{k'}$ for some $k' \leq k$ and $m\leq M_{k'}$. So we consider those $x$ belonging to some fixed $D_m^{k'}$. According to \eqref{ok for spanning Dk}, $D_m^{k'}\subset\!\subset B_{t_{m,k'}}(x_{m,k'})\subset \Omega$. Then since $D_m^{k'}$ has Lipschitz boundary, we are precisely in the position of being able to apply claim two from step one to deduce that $\mathcal{H}^n$-a.e. $x\in T[s] \cap D_m^{k'}$ belongs to $\partial^* \tilde{V}_i$ for some $\tilde{V}_i$ in the essential partition of $T$ induced by $\partial^* D_m^{k'} \cup T[s]$. Since $\partial^* D_m^{k'} \cup T[s] \subset S_k \cup T[s]$ by \eqref{partial Dk in Sk}, Theorem \ref{theorem decomposition}(a) is in force, and we conclude that $\mathcal{H}^n$-a.e. $x\in T[s] \cap D_m^{k'}$ belongs to $\cup_i \partial^* V_i$. This finishes the third case and therefore the proof of \eqref{spanning equation step four Sk}.
\par
\medskip
\noindent\textit{Step five}: As a preliminary computation before taking the limit in $k$ in the next step, we set $F = \cup_{k=0}^\infty \mathcal{D}_k$, and claim that it is a set of finite perimeter with
\begin{align}\label{small volume D}
    |F| &\leq \frac{\delta_j}{2}\quad \textup{and} \\ \label{compute the reduced boundary}
    \partial^* F \cap \Omega &\mbox{ is $\mathcal{H}^n$-equivalent to } \cup_k \partial^* \mathcal{D}_k\,.
\end{align}
Let us begin by observing that by \eqref{small volume Dk} (the volume estimate on $\mathcal{D}_k$), we have
\begin{align}\notag
    |F|\leq \lim_{k\to \infty} \sum_{k'=0}^k \frac{\delta_j}{2^{k'+2}} \leq \frac{\delta_j}{2}\,,
\end{align}
which is \eqref{small volume D}, and also
\begin{align}\label{fk to f}
    |F \Delta F_k| \to 0\quad\textup{as }k\to \infty\,.
\end{align}
By \eqref{per est Dk} and the fact that the sets $\cl \mathcal{D}_k$ are mutually pairwise disjoint, we have
\begin{align}\notag
     \limsup_{k\to \infty}P(F_k;\Omega) = \limsup_{k\to \infty}\sum_{k'=0}^k P(\mathcal{D}_k;\Omega) &\leq \limsup_{k\to \infty}\sum_{k'=0}^k 2\mathcal{H}^n(\mathcal{R}(K) \cap \cl \mathcal{D}_{k'} \setminus \partial^* E) + \frac{\delta_j}{2^{k'+2}} \\ \label{fk perimeter estimate}
    &\leq 2\mathcal{H}^n(\mathcal{R}(K)\setminus \partial^* E) + \frac{\delta_j}{2}\,,
\end{align}
which combined with the $L^1$ convergence of $F_k$ to $F$ implies that $F$ is a set of finite perimeter in $\Omega$. It remains to determine $\partial^* F\cap \Omega$. Now since $F= \cup_k \mathcal{D}_k$ and those sets are open, we have
\begin{align}\label{part F containment}
   \Omega \cap  \partial^* F \subset \Omega \cap \partial F \subset \cup_k \partial \mathcal{D}_k \shn \cup_k \partial^* \mathcal{D}_k\,,
\end{align}
where in the last containment we used the fact, recorded above \eqref{eating up k}, that $\mathcal{D}_k$ has Lipschitz boundary. To prove the reverse inclusion, we claim that it is enough to prove that
\begin{align}\label{quarter lower bound}
    \frac{1}{2}\leq \limsup_{r\to 0}\frac{|F \cap B_r(x)|}{|B_r(x)|}\leq \frac{3}{4}\quad\textup{for $\mathcal{H}^n$-a.e. $x\in \partial^* \mathcal{D}_k$, $k\geq 0$}\,.
\end{align}
Indeed, if \eqref{quarter lower bound} held, then since $\Omega \shn \Omega \cap( F^{(1)} \cup F^{(0)} \cup \partial^* F)$ and $\Omega \cap \partial^* F \overset{\mathcal{H}^n}{=}\Omega \cap \partial^e F$ by a theorem of Federer (see e.g. \cite[Theorem 3.61]{AFP}), we would thus have
$$
\cup_k \partial^* \mathcal{D}_k \shn\Omega \cap \partial^* F\,.
$$
Combining this inclusion with \eqref{part F containment} would complete the proof of \eqref{compute the reduced boundary}. Now suppose $x\in \partial^* \mathcal{D}_k$. By \eqref{away from Fk Dk}, there exists $r_x>0$ such that $B_{r_x}(x) \cap \mathcal{D}_{k'}=\varnothing$ for $k'<k$. Then by this avoidance and \eqref{boundary Dk estimate} applied for those $\mathcal{D}_{k'}$ with $k'>k$, we may estimate for $r<r_x$
\begin{align}\notag
  \frac{|\mathcal{D}_k \cap B_r(x)|}{|B_r(x)|} &\leq \frac{|F \cap B_r(x)|}{|B_r(x)|}\\ \notag
  &\leq \frac{\sum_{k'\geq k}|\mathcal{D}_{k'}\cap B_r(x)|}{|B_r(x)|} \leq \frac{|\mathcal{D}_k\cap B_r(x)|}{|B_r(x)|}+ \frac{\sum_{k'>  k}2^{-k'-3}|B_r(x)|}{|B_r(x)|} \leq \frac{|\mathcal{D}_k\cap B_r(x)|}{|B_r(x)|}+\frac{1}{4}\,.
\end{align}
Taking $r\to 0$ and using the fact that $x\in \partial^* \mathcal{D}_k$ gives \eqref{quarter lower bound}. 
\par
Before moving on to the final step, we record for later use that by \eqref{compute the reduced boundary} and \eqref{Dk no overlap with E}, 
\begin{align}\label{F no overlap with E}
    \mathcal{H}^n(\Omega \cap (\partial^* F \cap \partial^* E)) = 0\,.
\end{align}
Again by \cite[Exercise 16.5]{maggiBOOK}, this gives
\begin{align}\label{E delta F red bd}
    \Omega \cap \partial^* (E\Delta F) \overset{\mathcal{H}^n}{=}\Omega \cap [(\partial^* E \setminus \partial^* F)\cup(\partial^* F \setminus \partial^* E)] \overset{\mathcal{H}^n}{=}\Omega \cap (\partial^* E \cup \partial^* F)\,.
\end{align}
\par
\medskip
\noindent{\it Step six:} Our goal now is to apply the compactness Theorem \ref{theorem collapsed capillarity compactness} to the pairs $(S_k,E \Delta F_k)$ and then verify \eqref{need to show}. Now each $(S_k,E \Delta F_k)$ belongs to $\mathcal{K}_{\rm B}$ by \eqref{belongs to Kb} and we have
\begin{align}\notag
   \sup_j \mathcal{H}^n(S_k) \leq P(E;\Omega) + \sup_{k} P(F_k;\Omega) + \mathcal{H}^n(\mathcal{R}(K)) \overset{\eqref{fk perimeter estimate}}{<} \infty\,.
\end{align}
Note that $E\Delta F_k\overset{L^1}{\to} E\Delta F$ by \eqref{fk to f}. Therefore, Theorem \ref{theorem collapsed capillarity compactness} applies and says that the pair $(S,E\Delta F)$ belongs to $\mathcal{K}_{\rm B}$ and $S$ is $\mathcal{C}$-spanning, where
\begin{align}\notag
    S &= (\Omega \cap \partial^* (E\Delta F)) \cup \{x\in \Omega \setminus \partial^*(E\Delta F) : \theta_*^n(\mu)(x)\geq 2 \} \\ \notag
    &\overset{\eqref{E delta F red bd}}{=} (\Omega \cap (\partial^* E\cup \partial^*  F)) \cup \{x\in \Omega \setminus (\partial^*E\cup \partial^* F) : \theta_*^n(\mu)(x)\geq 2 \}\,,
\end{align}
where $\mu$ is the weak star limit of
\begin{align}\notag
    \mathcal{H}^n\mres (\Omega \cap \partial^* (E \Delta F_k)) &+ 2\mathcal{H}^n\mres (\mathcal{R}(S_k) \setminus \partial^* (E \Delta F_k)) \\ \notag
    &\overset{\eqref{application of maggi 16}}{=}\mathcal{H}^n\mres (\Omega \cap (\partial^* E \cup \partial^* F_k)) + 2\mathcal{H}^n\mres (\mathcal{R}(S_k) \setminus (\partial^* E \cup \partial^* F_k))\\ \notag
    &\overset{\eqref{computation of R}}{=}\mathcal{H}^n\mres (\Omega \cap (\partial^* E \cup \partial^* F_k)) + 2\mathcal{H}^n\mres (\mathcal{R}(K) \setminus (\partial^* E \cup F_k))\,.
\end{align}
Next, by \eqref{final eating up}, $\mathcal{H}^n(\mathcal{R}(K)\setminus(\partial^* E \cup F_k) )\to 0$, which implies then that $\mu$ is in fact the weak star limit of the measures $\mathcal{H}^n\mres (\Omega \cap (\partial^* E \cup \partial^* F_k))$. Therefore, for any $B_r(x)\subset\!\subset \Omega$ with $\mu(\partial B_r(x))=0$, 
\eqref{Fk no overlap with E}-\eqref{application of maggi 16} and \eqref{dk breakdown} give
\begin{align}\notag
   \mu(B_r(x)) =\lim_{k\to \infty}\mathcal{H}^n(\partial^* E \cup \partial^* F_k\cap B_r(x)) &= \lim_{k\to \infty} \mathcal{H}^n(\partial^* E\cap B_r(x)) + \sum_{k'=0}^k \mathcal{H}^n(\partial^* \mathcal{D}_{k'} \cap B_r(x)) \\ \notag
   &=  \mathcal{H}^n(\partial^* E\cap B_r(x)) + \sum_{k'=0}^\infty \mathcal{H}^n(\partial^* \mathcal{D}_{k'} \cap B_r(x))\,.
\end{align}
But by \eqref{F no overlap with E}-\eqref{E delta F red bd} and \eqref{compute the reduced boundary}, we also have
\begin{align}\label{what is the measure}
    \mathcal{H}^n  (B_r(x) \cap \partial^* (E \Delta F)) =\mathcal{H}^n(B_r(x) \cap \partial^* E)+ \mathcal{H}^n \Big(B_r(x) \cap  \bigcup_{k=0}^\infty \partial^* \mathcal{D}_k \Big)\,.
\end{align}
We have thus shown that $\mu(B_r(x)) = \mathcal{H}^n (B_r(x) \cap\partial^* (E \Delta F))$ for every $B_r(x)\subset\!\subset\Omega$ such that $\mu(\partial B_r(x))=0$. It follows that $\mu = \mathcal{H}^n\mres (\Omega \cap \partial^*(E \Delta F))$. Therefore, $S\setminus (\partial^* (E \Delta F))\overset{\mathcal{H}^n}{=}\varnothing$, and so $\Omega \cap \partial^* (E \Delta F)$ is $\mathcal{C}$-spanning since $S$ is $\mathcal{C}$-spanning.
\par
\medskip
We now have a set of finite perimeter $E \Delta F=:E_j$ with $\mathcal{C}$-spanning reduced boundary, and we must check \eqref{need to show}. For the volume estimate, we utilize \eqref{small volume D} to compute
\begin{align}\notag
    |E\Delta(E \Delta F)| = |E \cap F| + |F \setminus E| \leq 2|F| \leq \delta_j\,.
\end{align}
For the perimeter estimate, by \eqref{E delta F red bd} and \eqref{compute the reduced boundary}, \eqref{per est Dk}, and the fact that the $\mathcal{D}_k$ are pairwise disjoint,
\begin{align}\notag
P(E\Delta F;\Omega)&= P(E;\Omega) + \sum_{k=0}^\infty P(\mathcal{D}_k)\\ \notag
&\leq P(E;\Omega)+\sum_{k=0}^\infty2\,\mathcal{H}^n(\mathcal{R}(K) \cap \cl\mathcal{D}_k \setminus \partial^* E) + \frac{\delta_j}{2^{k+2}} \\ \notag
&\leq P(E;\Omega) + 2\,\mathcal{H}^n(\mathcal{R}(K) \setminus \partial^* E) + \frac{\delta_j}{2}\,, 
\end{align}
which is the desired perimeter bound in \eqref{need to show}.
\end{proof}

\begin{proof}[Proof of Theorem \ref{thm main of approximation} for $\Psi_{\rm bk}$]
    The proof can be accomplished by a simplification of the arguments in the boundary spanning case. First, as in step zero previously, it is enough to choose $\delta_j \to 0$ and produce a sequence of sets $\{E_j\}$ with $E_j^\one \cup (\Om \cap \pa^* E_j)$ $\mathcal{C}$-spanning such that 
\begin{align}\label{need to show bulk}
|E_j\Delta E|\leq \delta_j\quad \textup{and}\quad P(E_j;\Omega) \leq \mathcal{F}_{\rm bk}(K,E) + \delta_j\,.
\end{align}
Fix $\delta_j$, and let $A_0 \cup \cup_m f_m(A_m)$ be the decomposition of $\mathcal{R}(K)\cap E^\zero$ into an $\mathcal{H}^n$-null set and countably many compact Lipschitz graphs as in step two of the boundary case. Arguing as in step three, specifically the inequalities \eqref{small volume} and \eqref{closeness of Hausdorff business}, we may choose open sets $D_m$ containing $f_m(A_m)$ such that
\begin{align}\notag
|D_m| \leq 2^{-m}\delta_j\,,\quad P(D_m) \leq 2\mathcal{H}^n(f_m(A_m)) + 2^{-m}\delta_j \,.   
\end{align}    
We then set $E_j = E \cup \cup_m D_m$, so that the estimates
\begin{align}\notag
    |E_j \Delta E| \leq \sum_m |D_m| \leq \delta_j\,,\quad P(E_j;\Om) \leq P(E;\Om) + \sum_m P(D_m;\Om)\leq P(E;\Om) + \delta_j
\end{align}
follow. To see that $E_j^\one \cup (\Om \cap \pa^*E_j)$ is $\mathcal{C}$-spanning, we first note that, since supersets of $\mathcal{C}$-spanning sets are $\mathcal{C}$-spanning and property of being $\mathcal{C}$-spanning is stable under $\mathcal{H}^n$-null perturbations, it is enough to show that
\begin{align}\label{why is bulk spanning}
    E^\one \cup (\Om \cap \pa^* E) \cup (\mathcal{R}(K) \cap E^\zero) \shn E_j^\one \cup (\Om \cap \pa^* E_j)\,,
\end{align}
where the former set is $\mathcal{C}$-spanning by Theorem \ref{theorem spanning with partition}. Since $E_j$ is a superset of $E$, we have: first, that $E^\one \subset E_j^\one$; and second, that the Lebesgue density of $E_j$ at any $x\in \Om \cap \pa^* E$ is at least $1/2$. Therefore, by Federer's theorem, $E^\one \cup (\Om \cap \pa^* E)\shn E_j^\one \cup (\Om \cap \pa^* E_j)$. The inclusion $\mathcal{R}(K)\cap E^\zero\shn E_j^\one \cup (\Om \cap \pa^* E_j)$ follows directly from the inclusions $\mathcal{R}(K) \cap E^\zero \shn \cup_m f_m(A_m)$ and $f_m(A_m) \subset D_m^\one \subset E_j^\one$. 
\end{proof}

\subsection{Proof of Theorem \ref{thm main of minimality}} We can now show the equivalence of the collapsed and non-collapsed minimization problems.
\begin{proof}[Proof of Theorem \ref{thm main of minimality}]

For the equalities $\Psi_{\rm bd}(v)=\psi_{\rm bd}(v)$ and $\Psi_{\rm bk}(v)=\psi_{\rm bk}(v)$, first recall that $\Psi_{\rm bd}(v)\leq\psi_{\rm bd}(v)$ and $\Psi_{\rm bk}(v)\leq\psi_{\rm bk}(v)$, as noted in Section \ref{sec:previous results}. The reverse inequalities are direct consequences of Theorem \ref{thm main of approximation}, since any minimizing sequence for either relaxed problem \eqref{modified kms relax prob}/\eqref{modified kms relax bulk} can be approximated by a minimizing sequence for the corresponding soap film capillarity problem \eqref{soap film capillarity model}/\eqref{soap film capillarity model bulk}. It remains to show that, when $\wire$ is compact, given a minimizing sequence $\{E_j\}_j$ for $\psi_{\rm bd}(v)$ or $\psi_{\rm bk}(v)$, up to a subsequence we can extract $(K,E)\in \mathcal{K}_{\rm B}$ and ball $B_r(x)\subset(K \cup E \cup \wire)^c$ of volume $v-|E|$ such that $E_j\to E$ in locally in $L^1$,
\begin{align}\notag
    \mathcal{H}^n\mres(\Omega \cap \partial^* E_j) &\weakstar \mathcal{H}^n\mres(\Omega \cap \partial^* E) + 2\mathcal{H}^n\mres (K \setminus \partial^* E)\quad\mbox{or} \\ \notag
    \mathcal{H}^n\mres(\Omega \cap \partial^* E_j) &\weakstar \mathcal{H}^n\mres(\Omega \cap \partial^* E) + 2\mathcal{H}^n\mres (K \cap E^\zero)\,,
\end{align}
and $(K \cup \pa B_r(x),E \cup B_r(x))$ is minimal for $\Psi_{\rm bd}(v)$ or $\Psi_{\rm bk}(v)$.
The statement for the bulk problem $\psi_{\rm bk}(v)$ can be deduced as follows. By $\Psi_{\rm bk}(v) = \psi_{\rm bk}(v)$ and the fact that $E_j$ are admissible for $\Psi_{\rm bk}(v)$, $\{E_j\}_j$ is a minimizing sequence for $\Psi_{\rm bk}(v)$, and so \cite[Theorem 6.2]{MNR1}, which extracts such a ball and pair $(K,E)\in \mathcal{K}_{\rm B}$ out of an arbitrary minimizing sequence for $\Psi_{\rm bk}(v)$ such as $\{E_j\}_j$, yields the desired result. For $\psi_{\rm bd}(v)$, we can follow the same strategy. Let us sketch the argument, which, given the compactness in Theorem \ref{theorem collapsed capillarity compactness}, consists of standard arguments to deal with volume loss at infinity. By Theorem \ref{theorem collapsed capillarity compactness} and compactness for sets of finite perimeter, up to a subsequence, we obtain a limiting pair $(K,E)$ which is admissible for $\Psi_{\rm bd}(|E|)$ (where $v=0$ is $2\ell_{\rm B}$) such that $K$ is $\mathcal{C}$-spanning and 
\begin{align}\label{minni ineq}
    \mathcal{H}^n\mres (\Om \cap \pa^* E_j) \weakstar \mu \geq \mathcal{H}^n\mres (\Om \cap \pa^* E)  + 2\mathcal{H}^n\mres (K \setminus \pa^* E)\,.
\end{align}
Since $\wire$ is compact, the Euclidean isoperimetric inequality and \eqref{minni ineq} imply that the escaping mass contributes at least as much energy as a ball (see \cite[Proof of Theorem 6.2, step three]{MNR1}), that is
$$
\Psi_{\rm bd}(v) \geq \mathcal{F}_{\rm bd}(K,E) + P(B_r(x))
$$
whenever $|B_r(x)|=v-|E|$. However, by a construction consisting of adding balls of volume $v-|E|$ to any admissible $(K',E')$ for $\Psi_{\rm bd}(|E|)$ (see \cite[Proof of Theorem 6.2, step two]{MNR1}), we also have the reverse inequality
$$
\Psi_{\rm bd}(v) \leq \Psi_{\rm bd}(|E|) + P(B_r(x))\leq \mathcal{F}_{\rm bd}(K,E) + P(B_r(x))
$$
Combining these two, we find that $(K,E)$ is minimal for $\Psi_{\rm bd}(|E|)$, and so by Proposition \ref{rect part is spanning prop} and $\Om \cap \pa^* E \shn \mathcal{R}(K)$, we have $K\ehn \mathcal{R}(K)$. By a first variation argument \cite[Appendix C]{KMS1}, the integer multiplicity rectifiable varifold $V=\var(K,\theta)$, where $\theta=1$ on $\pa^* E\cap \Om$ and $\theta=2$ on $K \setminus \pa^* E$, has $L^\infty$-mean curvature in $\mathbb{R}^{n+1}\setminus \wire$. Therefore, by the monotonicity formula and the boundedness of $\wire$, $\spt\, V$ is bounded, and so are $\mathcal{R}(K)$ and $E$. Since $\wire \cup K \cup E$ is bounded and $\Psi_{\rm bd}(v) = \mathcal{F}_{\rm bd}(K,E) + P(B_r(x))$ whenever $|B_r(x)| = v - |E|$, we may choose $B_r(x)\subset (K \cup E \cup \wire)^c$ to add to $(K,E)$ and conclude the argument.
\end{proof}

\noindent{\bf Funding} This work was supported by National Science Foundation Grant DMS-2000034, National Science Foundation FRG Grant DMS-1854344, and National Science Foundation RTG Grant DMS-1840314. 

\medskip

\noindent{\bf Competing Interests} The author has no competing interests to declare that are relevant to the content of this article.

\medskip

\noindent{\bf Data Availability} There is no external data associated with this work.

\bibliographystyle{abbrv}
\bibliography{references}
\end{document}